\documentclass{daj}

\usepackage{enumitem}
\usepackage{amssymb}
\usepackage{amsthm}
\usepackage{mathtools}
\usepackage{bookmark}

\DeclareMathOperator{\rev}{rev}
\newcommand{\seq}{\subseteq}
\newcommand{\ol}{\overline}
\newcommand{\cC}{{\mathcal C}}
\newcommand{\cG}{{\mathcal G}}
\newcommand{\cH}{{\mathcal H}}
\newcommand{\cP}{{\mathcal P}}
\newcommand{\cR}{{\mathcal R}}
\newcommand{\cS}{{\mathcal S}}
\newcommand{\cT}{{\mathcal T}}

\newtheorem{theorem}{Theorem}

\newtheorem{proposition}[theorem]{Proposition}

\dajAUTHORdetails{%
  title = {A Short Proof of the Middle Levels Theorem}, 
  author = {Petr Gregor, Torsten M\"utze, and Jerri Nummenpalo},
  plaintextauthor = {Petr Gregor, Torsten Mutze, Jerri Nummenpalo},
  keywords = {Middle levels conjecture, Hamilton cycle, hypercube, vertex-transitive},
}   

\dajEDITORdetails{%
   year={2018},
   number={8},
   received={10 January 2018},   
   published={21 May 2018},  
   doi={10.19086/da.3659},       
}   

\begin{document}

\begin{frontmatter}[classification=text]

\title{A Short Proof of the Middle Levels Theorem} 

\author[pgregor]{Petr Gregor\thanks{Department of Theoretical Computer Science and Mathematical Logic, Charles University, 11800 Praha 1, Czech Republic, {\tt gregor@ktiml.mff.cuni.cz}}}
\author[tmuetze]{Torsten M\"utze\thanks{Institut für Mathematik, TU Berlin, 10623 Berlin, Germany, {\tt muetze@math.tu-berlin.de} }}
\author[njerri]{Jerri Nummenpalo\thanks{Department of Computer Science, ETH Zürich, 8092 Zürich, Switzerland, {\tt njerri@inf.ethz.ch}}}

\begin{abstract}
Consider the graph that has as vertices all bitstrings of length~$2n+1$ with exactly~$n$ or~$n+1$ entries equal to~1, and an edge between any two bitstrings that differ in exactly one bit.
The well-known middle levels conjecture asserts that this graph has a Hamilton cycle for any~$n\geq 1$.
In this paper we present a new proof of this conjecture, which is much shorter and more accessible than the original proof.
\end{abstract}
\end{frontmatter}

\section{Introduction}
\label{sec:intro}

The question whether a graph has a Hamilton cycle or not is one of the oldest and most fundamental problems in graph theory, with a wide range of practical applications.
Hamilton cycles are named after the Irish mathematician Sir William Rowan Hamilton, who lived in the 19th~century and who invented a puzzle that consists of finding such a cycle in the graph of the dodecahedron.
There are plenty of other families of highly symmetric graphs for which the existence of Hamilton cycles is a notoriously hard problem.
Consider e.g.\ the graph~$G_n$ that has as vertices all bitstrings of length~$2n+1$ with exactly~$n$ or~$n+1$ entries equal to~1, and an edge between any two bitstrings that differ in exactly one bit.
The graph~$G_n$ is a subgraph of the $(2n+1)$-dimensional hypercube, or equivalently, of the cover graph of the lattice of subsets of a $(2n+1)$-element ground set ordered by inclusion.
The well-known \emph{middle levels conjecture} asserts that $G_n$ has a Hamilton cycle for every~$n\geq 1$.
This conjecture is a special case of Lov{\'a}sz' conjecture on the Hamiltonicity of connected vertex-transitive graphs~\cite{MR0263646}, which can be considered the most far-ranging generalization of Hamilton's original puzzle.
The middle levels conjecture was raised in the 80s~\cite{MR737021,MR737262}, and has been attributed to Erd{\H{o}}s, Trotter and various others~\cite{MR962224}.
It also appears in the popular books~\cite{MR2034896,MR3444818,MR2858033} and in Gowers' recent expository paper~\cite{MR3584100}.
This seemingly innocent problem has attracted considerable attention over the last 30 years (see e.g.~\cite{MR1275228,MR1350586,MR1329390,MR1268348,MR2046083}), and a positive solution has been announced only recently.

\begin{theorem}[\cite{MR3483129}]
\label{thm:main}
For any~$n\geq 1$, the graph~$G_n$ has a Hamilton cycle.
\end{theorem}

The proof of Theorem~\ref{thm:main} given in~\cite{MR3483129} is long and technical (40 pages), so the main purpose of this paper is to give a shorter and more accessible proof.
This is achieved by combining ingredients developed in~\cite{MR3738156} with new ideas that allow us to avoid most of the technical obstacles in the original proof.
The new construction also yields the stronger result from~\cite{MR3483129} that the graph~$G_n$ has at least $\frac{1}{4}2^{2^{\lfloor (n+1)/4\rfloor}}=2^{2^{\Omega(n)}}$ different Hamilton cycles.
It also greatly simplifies the constant-time algorithm from~\cite{MR3627876} to generate each bitstring of the corresponding Hamilton cycle and several generalizations of it presented in~\cite{MR3758308}.
Since its first proof, Theorem~\ref{thm:main} has been used as an induction basis to prove several far-ranging generalizations, in particular Hamiltonicity of the bipartite Kneser graphs~\cite{MR3759914}, so our new proof also shortens this chain of arguments considerably.
Moreover, in two subsequent papers we apply the techniques developed here to resolve the case~$k=1$ of a generalized version of the middle levels conjecture where the vertex set of the underlying graph are all bitstrings with exactly~$w$ occurrences of~1 with $w\in\{n-k,\ldots,n+1+k\}$ \cite{jaeger-et-al:18} (the case~$k=0$ is the original conjecture), and to prove that the sparsest Kneser graphs~$K(2n+1,n)$, also known as odd graphs, have a Hamilton cycle for any~$n\geq 3$, settling an old conjecture from the 70s~\cite{muetze-nummenpalo-walczak:18}.

\subsection{Description of the Hamilton cycle}
\label{sec:def-hc}

We start right away by giving an explicit description of a Hamilton cycle in the graph~$G_n$.
The construction proceeds in two steps:
We first define a 2-factor in~$G_n$, i.e., a collection of disjoint cycles which together visit all vertices of the graph.
We then modify this 2-factor locally to join the cycles to a single cycle.

Specifically, the 2-factor~$\cC_n$ is defined as the union of two edge-disjoint perfect matchings in~$G_n$, namely the $(n-1)$-lexical and the $n$-lexical matching introduced in~\cite{MR962224}, which will be defined later.
The modification operation consists in taking the symmetric difference of~$\cC_n$ with a carefully chosen set of edge-disjoint 6-cycles.
Each 6-cycle used has the following properties: it shares two non-incident edges with one cycle~$C$ from the 2-factor~$\cC_n$, and one edge with a second cycle~$C'$ from the 2-factor, such that taking the symmetric difference between the edge sets of~$C,C'$ and the 6-cycle joins~$C$ and~$C'$ to one cycle, see Figure~\ref{fig:c6xy}.
Note that every 6-cycle in~$G_n$ can be described uniquely as a string~$x$ of length~$2n+1$ over the alphabet $\{0,1,*\}$ with $n-1$ occurrences of~1, $n-1$ occurrences of~0 and three occurrences of~$*$.
The 6-cycle corresponding to this string~$x$ is obtained by substituting the three occurences of~$*$ by all six combinations of symbols from~$\{0,1\}$ that use each symbol at least once.
We let~$D_i$ for~$i\ge 0$ denote the set of all bitstrings of length~$2i$ with exactly $i$ occurrences of~1 with the property that in every prefix, the number of 1-entries is at least as large as the number of 0-entries, and we define $D:=\bigcup_{i\geq 0} D_i$ as the set of all such Dyck words.
Let~$\cS_n$ denote the set of all 6-cycles in~$G_n$ encoded by strings of length~$2n+1$
\begin{equation}
\label{eq:c6}
  (u_1,0,u_2,0,\ldots,u_d,0,1,*,*,w,*,v_d,1,v_{d-1},1,\ldots,v_1,1,v_0,0)
\end{equation}
for some~$d\geq 0$ and $u_1,\ldots,u_d,v_0,\ldots,v_d,w\in D$.
We later prove that the 6-cycles from~$\cS_n$ are pairwise edge-disjoint and that this set contains a subset~$\cT_n\seq \cS_n$ such that the symmetric difference of the edge sets~$\cC_n\bigtriangleup \cT_n$ is a Hamilton cycle in~$G_n$.

\subsection{Proof outline}

After setting up some important definitions in Section~\ref{sec:def}, our proof of Theorem~\ref{thm:main} proceeds as follows:
We first establish crucial properties about the 2-factor~$\cC_n$ and about the set of 6-cycles~$\cS_n$ in Sections~\ref{sec:2factor} and \ref{sec:6cycles}, captured in Propositions~\ref{prop:2factor} and \ref{prop:6cycles}, respectively.
In Section~\ref{sec:proof} we combine these properties into the final proof.

\section{Preliminaries}
\label{sec:def}

\textbf{\textit{Bitstrings and Dyck paths.}}
Recall the definition of~$D_n$ from before.
We define the set~$D_n^-$ similarly, but we require that in exactly one prefix, the number of 1-entries is strictly smaller than the number of 0-entries.
We often interpret a bitstring~$x$ in~$D_n$ as a \emph{Dyck path} in the integer lattice~$\mathbb{Z}_2$ that starts at the origin and that consists of $n$~upsteps and $n$~downsteps that change the current coordinate by~$(+1,+1)$ or~$(+1,-1)$, respectively, corresponding to a~1 or a~0 in~$x$, see Figure~\ref{fig:xypair}.
By the prefix property, the corresponding lattice path has no steps below the abscissa.
Similarly, the lattice paths corresponding to bitstrings in~$D_n^-$ have exactly one downstep and one upstep below the abscissa.
We refer to a subpath of~$x$ from the set~$D$ as a \emph{hill} in~$x$.
Any bitstring~$x\in D_n$ can be written uniquely as $x=(1,u,0,v)$ with~$u,v\in D$.
We refer to this as the \emph{canonic decomposition} of~$x$.
For any bitstring~$x$, $\ol{\rev}(x)$ denotes the reversed and complemented bitstring.
In terms of lattice paths, $\ol{\rev}(x)$ is obtained by mirroring~$x$ at a vertical line.
The operation~$\ol{\rev}$ is applied to a sequence or a set of bitstrings by applying it to each entry or each element, respectively.
For a set of bitstrings~$X$ and a bitstring~$x$, we write~$X\circ x$ for the set obtained by concatenating each bitstring from~$X$ with~$x$.
The length of a sequence~$x$ is denoted by~$|x|$.

\textbf{\textit{Rooted trees and plane trees.}}
An~\emph{(ordered) rooted tree} is a tree with a specified root vertex, and the children of each vertex have a specified left-to-right ordering.
We think of a rooted tree as a tree embedded in the plane with the root on top, with downward edges leading from any vertex to its children, and the children appear in the specified left-to-right ordering.
Using a standard Catalan bijection, every Dyck path~$x\in D_n$ can be interpreted as a rooted tree with $n$~edges, see~\cite{MR3467982} and Figure~\ref{fig:xypair}.
Specifically, traversing the rooted tree starting at the root via a depth-first search, visiting the chilren in the specified left-to-right ordering, and writing an upstep for each visit of a child and a downstep for each return to the parent produces the corresponding Dyck path, and similarly vice versa.
A~\emph{rotation operation} moves the root to the leftmost child of the root, yielding another rooted tree, see Figure~\ref{fig:rot}.
Formally, in terms of Dyck paths, rotating the tree with canonic decomposition $(1,u,0,v)$, where~$u,v\in D$, yields the tree~$(u,1,v,0)$.
\emph{Plane trees} are obtained as equivalence classes of rooted trees under rotation, so they have no root, but a cyclic ordering of all neighbors at each vertex.

\textbf{\textit{Lexical matchings.}}
We recap the definition of the $(n-1)$-lexical and $n$-lexical matchings in~$G_n$ from~\cite{MR962224}.
We denote the two matchings as bijections $M,N:B_n\rightarrow B_n'$, where~$B_n$ and~$B_n'$ are the sets of bitstrings of length~$2n+1$ with exactly $n$ or $n+1$ occurrences of~1, respectively.
These sets are the two partition classes of the bipartite graph~$G_n$.
Given~$x\in B_n$, we sort all prefixes of~$x$ ending in~0 in decreasing order according to the surplus of the number of 0-entries compared to the number of 1-entries, breaking ties by sorting according to increasing lengths of the prefixes, yielding a total order on all these prefixes.
Then~$M(x)$ is obtained by flipping the last bit of the second prefix in this total order, and~$N(x)$ is obtained by flipping the last bit of the first prefix in this total order.
E.g., for $x=1101000$ the prefixes are ordered $1101000, 110100, 110, 11010$, so $M(x)=1101010$ and $N(x)=1101001$.
Clearly, $M(x)\neq N(x)$ for all~$x\in B_n$.
It is also easy to check that~$M$ and~$N$ are bijections.
In fact, $M^{-1}$ and $N^{-1}$ are obtained by considering prefixes ending in~1 and by changing only the secondary criterion in the above definition of a total order by sorting according to decreasing (instead of increasing) lengths of the prefixes.
It follows that~$M$ and~$N$ are edge-disjoint perfect matchings in~$G_n$, and their union is our 2-factor~$\cC_n=M\cup N$.

\section{Properties of the 2-factor}
\label{sec:2factor}

As adjacent vertices in~$G_n$ differ only in a single bit, every cycle from the 2-factor~$\cC_n$ can be described concisely by specifying a starting vertex on the cycle, and a sequence of bit positions to be flipped along the cycle until the starting vertex is reached again.
Proposition~\ref{prop:2factor} below states all relevant properties of the 2-factor~$\cC_n$ that we use, and in particular gives such a description of the bitflip sequences that are encountered when following each cycle from our 2-factor~$\cC_n$.
These sequences can be described nicely in terms of vertices of the form~$(x,0)$ where~$x\in D_n$.
Specifically, we define for any~$x\in D_n$ a \emph{bitflip sequence} $\sigma(x)$ as follows:
We consider the canonic decomposition $x=(1,u,0,v)$ and define $a:=1$, $b:=|u|+2$ and
\begin{subequations}
\label{eq:sigma-rec}
\begin{equation}
\label{eq:sigma}
  \sigma(x):=(b,\, a,\, \sigma_{a+1}(u)) \enspace,
\end{equation}
where~$\sigma_a(x')$ is defined for any substring~$x'\in D$ of~$x$ starting at position~$a$ in~$x$ by considering the canonic decomposition $x'=(1,u',0,v')$, by defining $b:=a+|u'|+1$ and by recursively computing
\begin{equation}
\label{eq:sigmaa}
  \sigma_a(x') := \begin{cases}
      () & \text{if } |x'|=0 \enspace, \\
      \big(b,\, a,\, \sigma_{a+1}(u'),\, a-1,\, b,\, \sigma_{b+1}(v')\big) & \text{otherwise} \enspace.
      \end{cases}
\end{equation}
\end{subequations}
Note that in these definitions, $a$ and $b$ are the positions of the first and last bit, respectively, of the substrings $(1,u,0)$ and $(1,u',0)$ in~$x$.
We denote by~$P_\sigma(x)$ the sequence of vertices in the $2n$-cube obtained by starting at the vertex~$x$ and flipping bits one after the other at the positions in the sequence~$\sigma(x)$.
We will prove in Proposition~\ref{prop:2factor} that $P_\sigma(x)\circ 0$ is in fact a path in the middle levels graph~$G_n$.
E.g., if $x=110100$, then we have $\sigma(x)=(6,1,3,2,1,3,5,4,3,5)$, so $P_\sigma(x)=(110100,110101,010101,011101,001101,\ldots,101001)$.

This definition has a straightforward interpretation in terms of Dyck paths.
In~\eqref{eq:sigma}, we consider the first hill $(1,u,0)$ of the Dyck path~$x$, first flip its last step (position~$b$), then its first step (position~$a$), and then recursively steps inside the hill.
In~\eqref{eq:sigmaa}, we consider the first hill $(1,u',0)$ of the Dyck path~$x'$, first flip its last step (position~$b$), then its first step (position~$a$), then recursively steps inside the hill, then the step to the left of the first step (position~$a-1$), then the last step again (position~$b$), and finally we recurse into the remaining part~$v'$.

\begin{proposition}
\label{prop:2factor}
For any~$n\geq 1$, the 2-factor~$\cC_n$ defined in Section~\ref{sec:def-hc} has the following properties:
\begin{enumerate}[label=(\roman*)]
\item Removing from~$\cC_n$ the edges that flip the last bit yields two sets of paths $\cP_n\circ 0$ and $\ol{\rev}(\cP_n)\circ 1$.
\item Each path from~$\cP_n$ starts at a vertex from~$D_n$ and ends at a vertex from~$D_n^-$.
The sets of all first and last vertices are~$D_n$ and~$D_n^-$, respectively.
\item For any path $P\in \cP_n$ and its first vertex $x\in D_n$ we have $P=P_\sigma(x)$ with~$\sigma$ defined in~\eqref{eq:sigma-rec}.
\item For any path $P\in \cP_n$, consider its first vertex $x\in D_n$ and last vertex $y\in D_n^-$.
If $x=(1,u,0,v)$ is the canonic decomposition of~$x$, then we have $y=(u,0,1,v)$.
Moreover, the distance between~$x$ and~$y$ along~$P$ is~$2|u|+2$.
\item For any cycle~$C\in \cC_n$, consider two vertices $(x,0),(y,0)$, where $x,y\in D_n$, that appear consecutively in the subsequence of all vertices of this form along~$C$.
If $x=(1,u,0,v)$ is the canonic decomposition of~$x$, then we have $y=(u,1,v,0)$ (or vice versa).
In terms of rooted trees, $y$ is obtained from~$x$ by a rotation operation.
Moreover, the distance between~$(x,0)$ and~$(y,0)$ along~$C$ is~$4n+2$.
\item The set of cycles of~$\cC_n$ is in bijection with the set of plane trees with $n$~edges.
\end{enumerate}
\end{proposition}

\begin{figure}
\centering
\makebox[0cm]{
\includegraphics[scale=0.916]{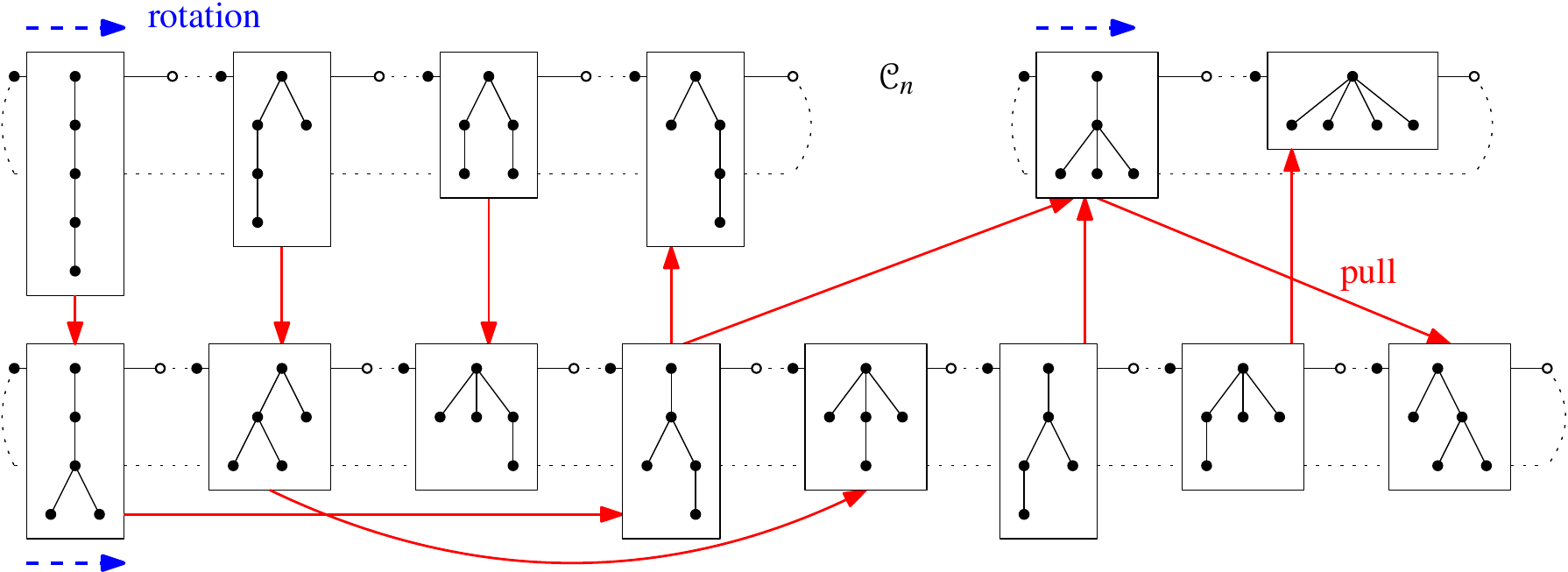}
}
\caption{Cycle structure of the 2-factor~$\cC_n$ and auxiliary graph~$\cH_n$ for~$n=4$.}
\label{fig:rot}
\end{figure}

The interpretation of the cycles of~$\cC_n$ in terms of rooted trees is illustrated in Figure~\ref{fig:rot} (ignore the solid arrows for the moment).

\begin{proof}
To prove~(i), let~$\cC_n^-$ denote the spanning subgraph of~$\cG_n$ obtained from~$\cC_n$ by removing the edges that flip the last bit.
As~$\cC_n$ is a union of cycles, $\cC_n^-$ is a union of paths $\cP_n\circ 0$, $\cP_n'\circ 1$ and possibly some cycles $\cR_n\circ 0$, $\cR_n'\circ 1$.
Consider the automorphism $f(x_1,\ldots,x_{2n+1}):=(\ol{\rev}(x_1,\ldots,x_{2n}),\ol{x_{2n+1}})$ of the graph~$G_n$.
It is easy to check that $f(M)=M$ and $f(N)=N$, implying that $\cP_n'=\ol{\rev}(\cP_n)$ and $\cR_n'=\ol{\rev}(\cR_n)$, so we have
\begin{equation}
\label{eq:Cn-}
  \cC_n^- = (\cP_n\cup \cR_n)\circ 0 \;\cup\; \ol{\rev}(\cP_n\cup \cR_n)\circ 1 \enspace.
\end{equation}
This almost proves~(i).
The only thing left to verify is that $\cR_n=\emptyset$, which will be done later.

To prove (ii)--(iv), consider an end vertex~$x$ of a path from~$\cP_n$.
It corresponds to a vertex $(x,0)\in B_n$ such that either~$M$ or~$N$ flips the last bit of~$(x,0)$.
By the definition of~$M$ and~$N$, this happens if and only if $x\in D_n^-$ or $x\in D_n$, respectively.
Consequently, the end vertices of~$\cP_n$ are given by $D_n\cup D_n^-$.

Now consider a path $P\in \cP_n$ with end vertex $x\in D_n$, and let $x=(1,u,0,v)$ be the canonic decomposition of~$x$.
We now show that $P=P_\sigma(x)$.
Note that every recursion step in the definition~\eqref{eq:sigma-rec} corresponds to a pair of indices $1\leq a<b\leq |u|+2$ in~$x$ such that $(x_a,\ldots,x_b)=(1,w',0)$ with $w'\in D$.
We refer to such a pair~$(a,b)$ as a \emph{base pair of~$x$}.
For any such base pair~$(a,b)$, we can partition~$x$ uniquely as
\begin{equation}
\label{eq:x}
  x=(1,u_1,1,u_2,\ldots,1,u_d,1,w',0,v_d,0,v_{d-1},0,\ldots,v_1,0,v)
\end{equation}
with~$d\geq 0$ and $u_1,\ldots,u_d,v_1,\ldots,v_d\in D$, see Figure~\ref{fig:xypair}.
Note that $a=1+\sum_{i=1}^d(1+|u_i|)$ and $b=a+|w'|+1$.
Let~$x'$ and~$x''$ denote the entries of the sequence~$P_\sigma(x)$ at positions~$2a-1$ and~$2b-1$, respectively.
These are well-defined vertices as~$\sigma(x)$ has length~$2|u|+2$ by definition~\eqref{eq:sigma-rec} and by the inequality $a<b\leq |u|+2$.
Using definition~\eqref{eq:sigma-rec}, a straightforward computation shows that for any base pair~$(a',b')$ and the corresponding substring $(1,u',0)\in D$ of~$x$, applying the bitflips in~$\sigma(x)$ to this substring, every bit~$x_i$ followed by $x_{i+1}=x_i$ is flipped twice, whereas every bit~$x_i$ followed by $x_{i+1}=\ol{x_i}$ is flipped once or three times, depending on whether~$x_i=1$ or~$x_i=0$, respectively.
This effectively shifts the bitstring to the left, yielding $(u',0,x_{b'+1})$.
Using this observation, the vertices~$x'$ and~$x''$ can be computed from~\eqref{eq:x} as
\begin{subequations}
\begin{align}
  x' &= (u_1,0,u_2,0,\ldots,u_d,0,1,w',0,v_d,1,v_{d-1},1,\ldots,v_1,1,v) \enspace, \label{eq:xa} \\
  x'' &= (u_1,0,u_2,0,\ldots,u_d,0,w',0,1,v_d,1,v_{d-1},1,\ldots,v_1,1,v) \enspace. \label{eq:xb}
\end{align}
\end{subequations}
By~\eqref{eq:sigma} and~\eqref{eq:sigmaa}, the next two bits flipped after~$x'$ are at positions~$b$ and~$a$.
Using~\eqref{eq:xa} and the definition of the mappings~$M$ and~$N^{-1}$, these are exactly the two bits flipped along the edge from~$M$ that starts at $(x',0)\in B_n$ and along the edge from~$N$ that starts at $M(x',0)\in B_n'$, respectively.
Similarly, if $b<|u|+2$, then by~\eqref{eq:sigmaa}, the next two bits flipped after~$x''$ are at positions~$a-1$ and~$b$.
Using~\eqref{eq:xb} and the definition of~$M$ and~$N^{-1}$, these are exactly the two bits flipped along the edge from~$M$ that starts at $(x'',0)\in B_n$ and along the edge from~$N$ that starts at $M(x'',0)\in B_n'$, respectively.
As this argument holds for all base pairs~$(a,b)$ of~$x$, we obtain $P=P_\sigma(x)$, proving~(iii).
Applying~\eqref{eq:xb} for the base pair $(a,b)=(1,|u|+2)$ of~$x$ (in this case~$d=0$ and~$w'=u$), the last vertex~$y$ reached on the path $P=P_\sigma(x)$ is $y=(u,0,1,v)\in D_n^-$.
This proves~(ii).
Recall that $|\sigma(x)|=2|u|+2$, so the distance between~$x$ and~$y$ along~$P$ is~$2|u|+2$, proving~(iv).

To prove~(v), consider a path~$P\in \cP_n$ with first vertex $x=(1,u,0,v)\in D_n$, where $u,v\in D$, and last vertex $y':=(u,0,1,v)\in D_n^-$.
We consider the cycle $C\in \cC_n$ containing the path~$P\circ 0$ and continue to follow this cycle.
The next edge of~$C$ after traversing~$P\circ 0$ flips the last bit, so from~$(y',0)$ we reach the vertex~$(y',1)$.
By~\eqref{eq:Cn-}, the path traversed by~$C$ until the last bit is flipped again is $\ol{\rev}(P')\circ 1$ for some $P'\in \cP_n$.
As the last vertex of~$P'$ is $\ol{\rev}(y')=(\ol{\rev}(v),0,1,\ol{\rev}(u))\in D_n^-$, its first vertex is $x':=(1,\ol{\rev}(v),0,\ol{\rev}(u))\in D_n$ by~(iv).
As the path $\ol{\rev}(P')\circ 1$ is traversed backwards by~$C$, the next vertex on~$C$ after traversing $\ol{\rev}(P')\circ 1$ is $(y,0)$ with $y:=\ol{\rev}(x')=(u,1,v,0)\in D_n$.
The distance between~$(x,0)$ and~$(y,0)$ along~$C$ is $(2|u|+2)+(2|v|+2)+2$ by (iv), which equals $2(|u|+|v|+2)+2=4n+2$.
This almost proves (v), assuming that $\cR_n=\emptyset$ in~\eqref{eq:Cn-}.
However, the total number of vertices visited by the paths $\cP_n\circ 0$ and $\ol{\rev}(\cP_n)\circ 1$ is~$(4n+2)|D_n|$.
As the cardinality of~$D_n$ is given by the $n$-th Catalan number~\cite{MR3467982}, this quantity equals~$2\binom{2n+1}{n}$, the total number of vertices of~$G_n$.
It follows that $\cR_n=\emptyset$ in~\eqref{eq:Cn-}, completing the proofs of~(i) and~(v).

Claim~(vi) is an immediate consequence of~(ii), (v), and the definition of plane trees.
\end{proof}

\section{Properties of the 6-cycles}
\label{sec:6cycles}

Proposition~\ref{prop:6cycles} below states all relevant properties of the set of 6-cycles~$\cS_n$ that we use.
To state the proposition, we say that $x,y \in D_n$ form a \emph{flippable pair}~$(x,y)$, if
\begin{equation}
\label{eq:xy}
\begin{aligned}
  x &= (1,u_1,1,u_2,\ldots,1,u_d,1,1,0,w,0,v_d,0,v_{d-1},0,\ldots,v_1,0,v_0) \enspace, \\
  y &= (1,u_1,1,u_2,\ldots,1,u_d,1,0,1,w,0,v_d,0,v_{d-1},0,\ldots,v_1,0,v_0)
\end{aligned}
\end{equation}
for some~$d\geq 0$ and $u_1,\ldots,u_d,v_0,\ldots,v_d,w\in D$.
In terms of rooted trees, the tree~$y$ is obtained from~$x$ by moving a pending edge from a vertex in the left subtree to its predecessor, see Figure~\ref{fig:xypair}.
We refer to $(1,1,0,w,0)$ and $(1,0,1,w,0)$ as \emph{flippable substrings of~$x$ and~$y$} corresponding to this flippable pair.
The corresponding subpaths are highlighted with gray boxes in the figure.
Note that a bitstring~$x$ may appear in multiple flippable pairs, as it may contain multiple flippable substrings.

\begin{figure}
\centering
\makebox[0cm]{
\includegraphics[scale=0.916]{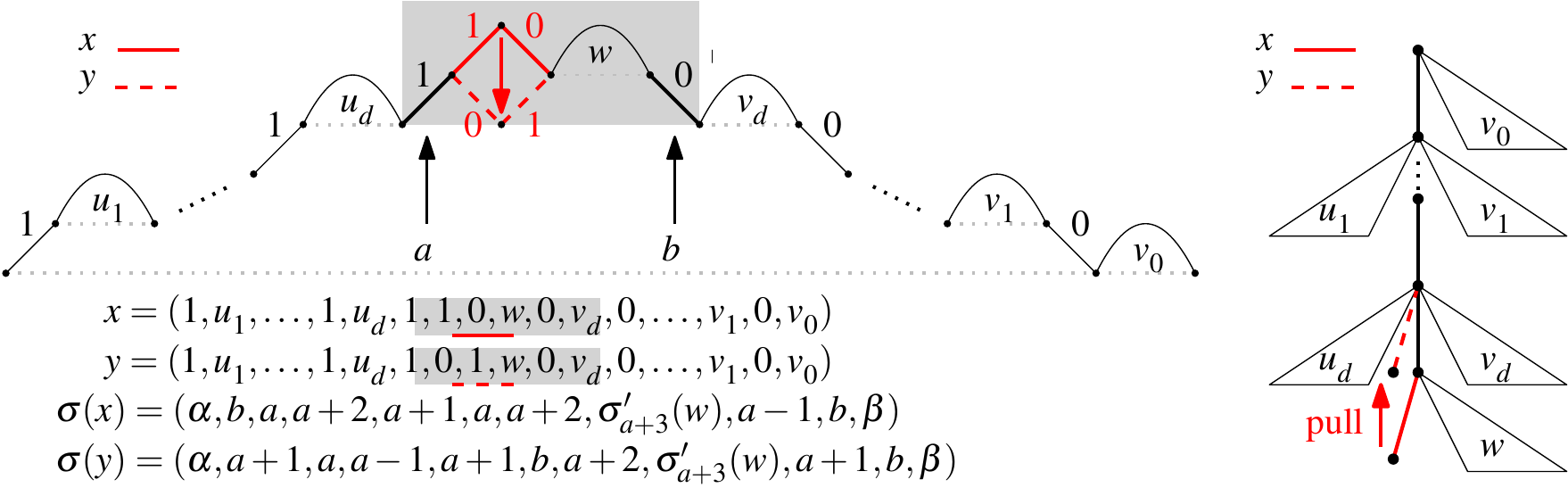}
}
\caption{A~flippable pair~$(x,y)$, its Dyck path interpretation (left) and rooted tree interpretation (right).}
\label{fig:xypair}
\end{figure}

Clearly, the set of 6-cycles~$\cS_n$ defined in Section~\ref{sec:def-hc} is given by considering all flippable pairs~$(x,y)$, $x,y\in D_n$, as in~\eqref{eq:xy}, by defining
\begin{equation}
\label{eq:c6xy}
  C_6(x,y):=(u_1,0,u_2,0,\ldots,u_d,0,1,*,*,w,*,v_d,1,v_{d-1},1,\ldots,v_1,1,v_0)
\end{equation}
and by taking the union of all 6-cycles~$C_6(x,y)\circ 0$.
Note here that~\eqref{eq:c6} and~\eqref{eq:c6xy} differ only in the additional 0-bit in the end.
In particular, all 6-cycles~$\cS_n$ that we use to join the cycles in the 2-factor~$\cC_n$ belong to the subgraph of~$G_n$ given by all vertices whose last bit equals~0.

\begin{proposition}
\label{prop:6cycles}
For any~$n\geq 1$, the 6-cycles~$C_6(x,y)$ defined in~\eqref{eq:c6xy} have the following properties:
\begin{enumerate}[label=(\roman*)]
\item
Let~$(x,y)$ be a flippable pair.
The symmetric difference of the edge sets of the two paths~$P_\sigma(x)$ and~$P_\sigma(y)$ with the 6-cycle~$C_6(x,y)$ gives two paths~$P'(x)$ and~$P'(y)$ on the same set of vertices as~$P_\sigma(x)$ and~$P_\sigma(y)$, where $P'(x)$ starts at~$x$ and ends at the last vertex of~$P_\sigma(y)$, and $P'(y)$ starts at~$y$ and ends at the last vertex of~$P_\sigma(x)$.

\item
Let~$(x,y)$ be a flippable pair and let~$a$ be the starting position of the corresponding flippable substring in~$x$.
The 6-cycle~$C_6(x,y)$ intersects~$P_\sigma(x)$ in the $(2a-1)$-th and the $(2a+4)$-th edge, and it intersects~$P_\sigma(y)$ in the $(2a-1)$-th edge.

\item
For any flippable pairs~$(x,y)$ and~$(x',y')$, the 6-cycles~$C_6(x,y)$ and~$C_6(x',y')$ are edge-disjoint.

\item
For any flippable pairs~$(x,y)$ and~$(x,y')$, the two pairs of edges that the two 6-cycles~$C_6(x,y)$ and~$C_6(x,y')$ have in common with the path~$P_\sigma(x)$ are not interleaved, but one pair appears before the other pair along the path.
\end{enumerate}
\end{proposition}

Informally, the first property asserts that a 6-cycle from~$\cS_n$ can be used to join two cycles from the 2-factor~$\cC_n$ to a single cycle, see Figure~\ref{fig:c6xy}.
The last two properties ensure that no two 6-cycles interfere with each other when iterating this joining operation.

\begin{figure}
\centering
\makebox[0cm]{
\includegraphics[scale=0.916]{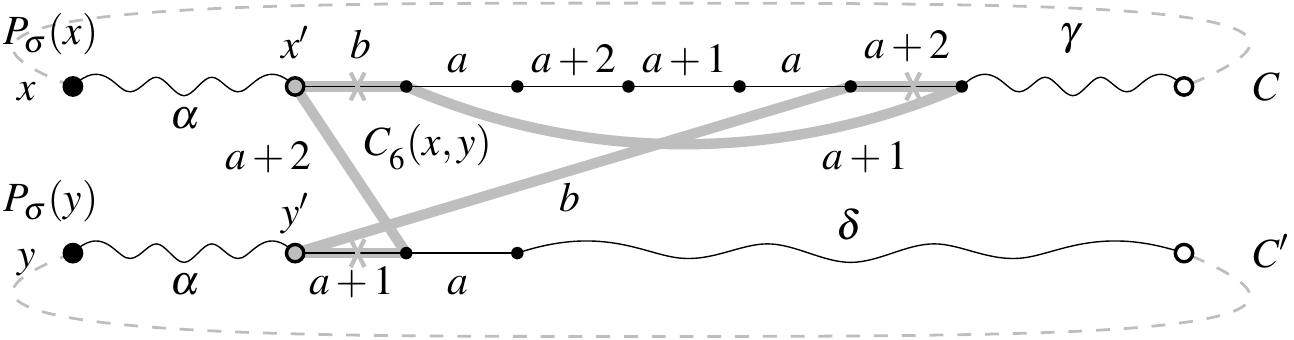}
}
\caption{
Two cycles from our 2-factor joined by taking the symmetric difference with a 6-cycle.
The paths~$P_\sigma(x)$ and~$P_\sigma(y)$ (solid black) lying on the two cycles traverse the 6-cycle~$C_6(x,y)$ (solid gray) as shown.
The symmetric difference yields paths $P'(x)=P(x,\tau(x))$ and $P'(y)=P(y,\tau(y))$ that have flipped end vertices.
}
\label{fig:c6xy}
\end{figure}

\begin{proof}
To prove (i), consider a flippable pair~$(x,y)$ as in~\eqref{eq:xy}, and let~$a$ and~$b$ be the first and last position of the corresponding flippable substring $(1,1,0,w,0)$ in~$x$, respectively.
Applying the definition~\eqref{eq:sigma-rec}, a straightforward computation yields the bitflip sequences
\begin{subequations}
\begin{equation}
\label{eq:sigmaxy}
\begin{aligned}
  \sigma(x) &= (\alpha,b,a,a+2,a+1,a,a+2,\gamma) \enspace, \\
  \sigma(y) &= (\alpha,a+1,a,\delta) \enspace,
\end{aligned}
\end{equation}
where if~$d=0$ we define
\begin{equation}
  \alpha:=\beta:=\delta:=() \;\; \text{and} \;\; \gamma:=(\sigma_{a+3}(w)) \enspace.
\end{equation}
On the other hand, if~$d>0$ then~$\alpha$ is the longest common prefix of~$\sigma(x)$ and~$\sigma(y)$, $(b,\beta)$ is their longest common suffix, and
\begin{equation}
\begin{aligned}
  \gamma:=(\sigma_{a+3}(w),a-1,b,\beta) \;\; \text{and} \;\;
  \delta:=(a-1,a+1,b,a+2,\sigma_{a+3}(w),a+1,b,\beta) \enspace.
\end{aligned}
\end{equation}
\end{subequations}
Note that $(\alpha,\beta)=\sigma(1,u_1,1,u_2,\ldots,1,u_d,v_d,0,v_{d-1},0,\ldots,v_1,0)$ and that $|\alpha|=2(d+\sum_{i=1}^d |u_i|)=2(a-1)=2a-2$.
The last relation expresses that we count two flip operations for each of the steps from the hills $u_1,u_2,\ldots,u_d$, one flip for each of the $d$~upsteps preceding the hills~$u_i$, and one flip for each of the $d$~downsteps following the hills~$v_i$.
Specifically, the vertices~$x'$ and~$y'$ that are reached from~$x$ or~$y$ by flipping all $2a-2$ bit positions in the sequence~$\alpha$ are
\begin{equation}
\label{eq:xyp}
\begin{aligned}
  x' &= (u_1,0,u_2,0,\ldots,u_d,0,1,1,0,w,0,v_d,1,v_{d-1},1,\ldots,v_1,1,v_0) \enspace, \\
  y' &= (u_1,0,u_2,0,\ldots,u_d,0,1,0,1,w,0,v_d,1,v_{d-1},1,\ldots,v_1,1,v_0) \enspace.
\end{aligned}
\end{equation}
Comparing~\eqref{eq:c6xy} and~\eqref{eq:xyp} shows that these vertices belong to the 6-cycle~$C_6(x,y)$.
From~\eqref{eq:sigmaxy} we observe that the 6-cycle~$C_6(x,y)$ is then traversed as depicted in Figure~\ref{fig:c6xy}.
In particular, $x'$ and~$y'$ are the first vertices from the paths~$P_\sigma(x)$ and~$P_\sigma(y)$ hitting the 6-cycle.
By taking the symmetric difference of these edge sets, we obtain paths~$P'(x)$ and~$P'(y)$ on the same vertex set as~$P_\sigma(x)$ and~$P_\sigma(y)$ with flipped end vertices.
Formally, $P'(x)$ and~$P'(y)$ are obtained by starting at~$x$ and~$y$ and flipping bits according to the modified bitflip sequences
\begin{align*}
  \tau(x) &:= (\alpha,a+2,a,\delta) \enspace, \\
  \tau(y) &:= (\alpha,b,a,a+1,a+2,a,a+1,\gamma) \enspace,
\end{align*}
respectively.
This proves~(i).

Recall from the previous argument that the distance between~$x$ and~$x'$ along the path~$P_\sigma(x)$ is $|\alpha|=2a-2$, and the same holds for the distance between~$y$ and~$y'$ along the path~$P_\sigma(y)$.
The 6-cycle~$C_6(x,y)$ intersects~$P_\sigma(x)$ in the next edge after~$x'$ and in the edge that is five edges further away, and it intersects~$P_\sigma(y)$ in the next edge after~$y'$.
Combining these facts proves~(ii).

To prove (iii), consider two 6-cycles~$C_6(x,y)$ and~$C_6(x',y')$.
Instead of comparing them directly, we consider how they intersect a fixed path~$P_\sigma(z)$ with $z\in \{x,y\}\cap\{x',y'\}$.
This is possible because all edges of these 6-cycles either lie on such a path or they go between two such paths.
Consider the two flippable substrings of~$z$ corresponding to~$C_6(x,y)$ and~$C_6(x',y')$ starting at positions~$a$ and~$a'$ in~$z$, respectively.
We assume w.l.o.g.\ that~$a'\geq a+1$.

We first consider the case~$z=y$ and $z\in\{x',y'\}$.
By~(ii) we know that the 6-cycle~$C_6(x,y)$ intersects the path~$P_\sigma(y)$ in the edge~$2a-1$.
However, we also have $2a'-1\geq 2(a+1)-1=2a+1$, so the edge(s) that the cycle~$C_6(x',y')$ has in common with~$P_\sigma(y)$ are separated by at least one edge along the path, proving that the two 6-cycles do not share any vertices on this path.

We now consider the case~$z=x$ and $z\in\{x',y'\}$.
By~(ii) we know that the 6-cycle~$C_6(x,y)$ intersects the path~$P_\sigma(x)$ in the edges~$2a-1$ and~$2a+4$.
If $a'\geq a+4$, then we have $2a'-1\geq 2(a+4)-1=2a+7$, so the edges that the cycle~$C_6(x',y')$ has in common with~$P_\sigma(x)$ are separated by at least two edges along the path, proving that the two 6-cycles do not share any vertices on this path.
It remains to consider the subcases $a'\in\{a+1,a+2,a+3\}$.
The case $a'=a+2$ can be excluded, because this would mean that~$x$ has a 0-bit at position~$a+2$ and~$x$ has a 1-bit at position~$a'=a+2$ by~\eqref{eq:xy}, which is a contradiction.
If $a'=a+1$, then since~$x$ has a 0-bit at position~$a+2$, it follows from~\eqref{eq:xy} that~$y'=x$ and that the flippable substring of~$x$ corresponding to~$(x,y)$ has the form $(1,1,0,w,0)=(1,1,0,1,w',0,0)$.
Consequently, by~(ii) $C_6(x',y')$ intersects the path~$P_\sigma(x)$ in the edge $2a'-1=2(a+1)-1=2a+1$, which is separated by at least one edge from both edges~$2a-1$ and~$2a+4$, so the two 6-cycles do not share any vertices on this path.
If $a'=a+3$, then either of the two cases~$x'=x$ or~$y'=x$ can occur, and in both cases the cycle~$C_6(x',y')$ intersects~$P_\sigma(x)$ in the edge $2a'-1=2(a+3)-1=2a+5$, and if~$x'=x$ also in the edge $2a'+4 = 2a+10$ (which is safe for sure).
The edge~$2a+5$ is different from the edge~$2a+4$ on~$P_\sigma(x)$, but both share an end vertex, so the other two edges of the 6-cycles~$C_6(x,y)$ and~$C_6(x',y')$ starting at this vertex and not belonging to~$P_\sigma(x)$ could be identical.
However, this is not the case as the corresponding edge from~$C_6(x,y)$ leads back to~$P_\sigma(x)$, whereas the corresponding edge from~$C_6(x',y')$ leads to~$P_\sigma(y')$ if~$x'=x$ and to~$P_\sigma(x')$ if~$y'=x$.

This completes the proof of (iii).

The previous analysis in the last case where $z=x=x'$ also proves (iv).
\end{proof}

\section{Proof of Theorem~\ref{thm:main}}
\label{sec:proof}

With Propositions~\ref{prop:2factor} and \ref{prop:6cycles} in hand, we are now ready to prove Theorem~\ref{thm:main}.

\begin{proof}[Proof of Theorem~\ref{thm:main}]
Let~$\cC_n$ and~$\cS_n$ be the 2-factor and the set of 6-cycles defined in Section~\ref{sec:def-hc}.

Consider two different cycles $C,C'\in \cC_n$ containing paths $P\circ 0\seq C$ and $P'\circ 0\seq C'$, where $P,P'\in\cP_n$, with first vertices $x,y\in D_n$, respectively, such that~$(x,y)$ is a flippable pair.
By Proposition~\ref{prop:6cycles}~(i), the symmetric difference of the edge sets $(C\cup C')\bigtriangleup (C_6(x,y)\circ 0)$ forms a single cycle on the same vertex set as~$C\cup C'$, i.e., this joining operation reduces the number of cycles in the 2-factor by one, see Figure~\ref{fig:c6xy}.
Recall from~\eqref{eq:xy} that in terms of rooted trees, the tree~$y$ is obtained from~$x$ by moving a pending edge from a vertex in the left subtree to its predecessor.
We refer to this as a \emph{pull operation}, see Figure~\ref{fig:xypair}.

We repeat this joining operation until all cycles in the 2-factor are joined to a single Hamilton cycle.
For this purpose we define an auxiliary graph~$\cH_n$ whose nodes represent the cycles in the 2-factor~$\cC_n$ and whose edges connect pairs of cycles that can be connected to a single cycle with such a joining operation that involves a 6-cycle from the set~$\cS_n$, see Figure~\ref{fig:rot}.
Formally, the node set of~$\cH_n$ is given by partitioning the set of all rooted trees with $n$~edges into equivalence classes under tree rotation.
By Proposition~\ref{prop:2factor}~(v) and (vi), each cycle~$C$ of~$\cC_n$ can be identified with one equivalence class under tree rotation, so the nodes of~$\cH_n$ indeed correspond to the cycles in the 2-factor~$\cC_n$.
Specifically, each rooted tree belonging to some node of~$\cH_n$ equals the first vertex $x\in D_n$ of some path $P\in \cP_n$ such that~$P\circ 0$ lies on the cycle corresponding to that node.
For every flippable pair~$(x,y)$, $x,y\in D_n$, we add the edge to~$\cH_n$ that connects the node containing the tree~$x$ to the node containing the tree~$y$.
In Figure~\ref{fig:rot}, those edges are drawn as solid arrows directed from~$x$ to~$y$.
By our initial argument, such a flippable pair yields a 6-cycle~$C_6(x,y)$ that can be used in~$G_n$ to join the two corresponding cycles to a single cycle.
Note that~$\cH_n$ may contain multiple edges or loops.

To complete the proof, it therefore suffices to prove that the graph~$\cH_n$ is connected.
Indeed, if~$\cH_n$ is connected, then we can pick a spanning tree in~$\cH_n$, corresponding to a collection of 6-cycles $\cT_n\seq \cS_n$, such that the symmetric difference between the edge sets $\cC_n\bigtriangleup \cT_n$ forms a Hamilton cycle in~$G_n$.
Here we need properties (iii) and (iv) in Proposition~\ref{prop:6cycles}, which ensure that whatever subset of 6-cycles we use in this joining process, they will not interfere with each other, guaranteeing that inserting each 6-cycle indeed reduces the number of cycles by one, as desired.

At this point we have reduced the problem of proving that~$G_n$ has a Hamilton cycle to showing that the auxiliary graph~$\cH_n$ is connected, which is much easier.
Indeed, all we need to show is that any rooted tree with $n$~edges can be transformed into any other tree by a sequence of rotations and pulls, and their inverse operations.
Recall that rotations correspond to following the same cycle from~$\cC_n$ (staying at the same node in~$\cH_n$), and a pull corresponds to a joining operation (traversing an edge in~$\cH_n$ to another node).
For this we show that any rooted tree~$x$ can be transformed into the special tree $s:=(1,1,0,1,0,\ldots,1,0,0)\in D_n$, i.e., a star with $n$~rays rooted at a leaf, by a sequence of rotations and pulls.
This can be achieved by rotating~$x$ until it is rooted at a leaf.
Now the left subtree is the entire tree, so we can repeatedly pull pending edges towards the unique child of the root until we end up at the star~$s$.

This completes the proof.
\end{proof}

\bibliographystyle{alpha}
\bibliography{../refs}

\newcommand{\etalchar}[1]{$^{#1}$}
\begin{thebibliography}{GJM{\etalchar{+}}18}

\bibitem[BW84]{MR737262}
M.~Buck and D.~Wiedemann.
\newblock Gray codes with restricted density.
\newblock {\em Discrete Math.}, 48(2-3):163--171, 1984.

\bibitem[DG12]{MR2858033}
P.~Diaconis and R.~Graham.
\newblock {\em Magical mathematics}.
\newblock Princeton University Press, Princeton, NJ, 2012.
\newblock The mathematical ideas that animate great magic tricks, With a
  foreword by Martin Gardner.

\bibitem[DKS94]{MR1268348}
D.~A. Duffus, H.~A. Kierstead, and H.~S. Snevily.
\newblock An explicit {$1$}-factorization in the middle of the {B}oolean
  lattice.
\newblock {\em J. Combin. Theory Ser. A}, 65(2):334--342, 1994.

\bibitem[FT95]{MR1350586}
S.~Felsner and W.~T. Trotter.
\newblock Colorings of diagrams of interval orders and {$\alpha$}-sequences of
  sets.
\newblock {\em Discrete Math.}, 144(1-3):23--31, 1995.
\newblock Combinatorics of ordered sets (Oberwolfach, 1991).

\bibitem[GJM{\etalchar{+}}18]{jaeger-et-al:18}
P.~Gregor, S.~J{\"a}ger, T.~M{\"u}tze, J.~Sawada, and K.~Wille.
\newblock Gray codes and symmetric chains.
\newblock To appear in {\it Proceedings of the 45th International Colloqium on
  Automata, Languages and Programming (ICALP 2018)}. {\it arXiv:1802.06021},
  2018.

\bibitem[GM18]{MR3758308}
P.~Gregor and T.~M\"utze.
\newblock Trimming and gluing {G}ray codes.
\newblock {\em Theoret. Comput. Sci.}, 714:74--95, 2018.

\bibitem[Gow17]{MR3584100}
W.~T. Gowers.
\newblock Probabilistic combinatorics and the recent work of {P}eter {K}eevash.
\newblock {\em Bull. Amer. Math. Soc. (N.S.)}, 54(1):107--116, 2017.

\bibitem[Hav83]{MR737021}
I.~Havel.
\newblock Semipaths in directed cubes.
\newblock In {\em Graphs and other combinatorial topics ({P}rague, 1982)},
  volume~59 of {\em Teubner-Texte Math.}, pages 101--108. Teubner, Leipzig,
  1983.

\bibitem[Joh04]{MR2046083}
J.~R. Johnson.
\newblock Long cycles in the middle two layers of the discrete cube.
\newblock {\em J. Combin. Theory Ser. A}, 105(2):255--271, 2004.

\bibitem[Knu11]{MR3444818}
D.~E. Knuth.
\newblock {\em The Art of Computer Programming. {V}ol. 4{A}. {C}ombinatorial
  Algorithms. {P}art 1}.
\newblock Addison-Wesley, Upper Saddle River, NJ, 2011.

\bibitem[KT88]{MR962224}
H.~A. Kierstead and W.~T. Trotter.
\newblock Explicit matchings in the middle levels of the {B}oolean lattice.
\newblock {\em Order}, 5(2):163--171, 1988.

\bibitem[Lov70]{MR0263646}
L.~Lov{\'a}sz.
\newblock Problem 11.
\newblock In {\em Combinatorial Structures and Their Applications (Proc.
  Calgary Internat. Conf., Calgary, Alberta, 1969)}. Gordon and Breach, New
  York, 1970.

\bibitem[MN17]{MR3627876}
T.~M{\"u}tze and J.~Nummenpalo.
\newblock A constant-time algorithm for middle levels {G}ray codes.
\newblock In {\em Proceedings of the {T}wenty-{E}ighth {A}nnual {ACM}-{SIAM}
  {S}ymposium on {D}iscrete {A}lgorithms}, pages 2238--2253. SIAM,
  Philadelphia, PA, 2017.

\bibitem[MNW18]{muetze-nummenpalo-walczak:18}
T.~M{\"u}tze, J.~Nummenpalo, and B.~Walczak.
\newblock Sparse {K}neser graphs are {H}amiltonian.
\newblock To appear in {\it Proceedings of the 50th Annual ACM Symposium on the
  Theory of Computing (STOC 2018)}. {\it arXiv:1711.01636}, 2018.

\bibitem[MS17]{MR3759914}
T.~M\"utze and P.~Su.
\newblock Bipartite {K}neser graphs are {H}amiltonian.
\newblock {\em Combinatorica}, 37(6):1207--1219, 2017.

\bibitem[MSW18]{MR3738156}
T.~M\"utze, C.~Standke, and V.~Wiechert.
\newblock A minimum-change version of the {C}hung-{F}eller theorem for {D}yck
  paths.
\newblock {\em European J. Combin.}, 69:260--275, 2018.

\bibitem[M{\"u}t16]{MR3483129}
T.~M{\"u}tze.
\newblock Proof of the middle levels conjecture.
\newblock {\em Proc. Lond. Math. Soc.}, 112(4):677--713, 2016.

\bibitem[Sav93]{MR1275228}
C.~D. Savage.
\newblock Long cycles in the middle two levels of the {B}oolean lattice.
\newblock {\em Ars Combin.}, 35(A):97--108, 1993.

\bibitem[Sta15]{MR3467982}
R.~P. Stanley.
\newblock {\em Catalan numbers}.
\newblock Cambridge University Press, New York, 2015.

\bibitem[SW95]{MR1329390}
C.~D. Savage and P.~Winkler.
\newblock Monotone {G}ray codes and the middle levels problem.
\newblock {\em J. Combin. Theory Ser. A}, 70(2):230--248, 1995.

\bibitem[Win04]{MR2034896}
P.~Winkler.
\newblock {\em Mathematical puzzles: a connoisseur's collection}.
\newblock A K Peters, Ltd., Natick, MA, 2004.

\end{thebibliography}

\end{document}